\RequirePackage{fix-cm}
\RequirePackage{amsmath}
\RequirePackage[hyphens]{url}
\documentclass[5p,times]{elsarticle}
\journal{Operations Research Letters}
\usepackage{amssymb}
\usepackage{placeins}
\usepackage{xcolor}
\usepackage{mathtools}
\usepackage{mathrsfs}
\usepackage{amsfonts}
\usepackage{booktabs}
\usepackage{amsthm}

\usepackage[hidelinks,hypertexnames=false,colorlinks=true,breaklinks=true,bookmarks=true,urlcolor=blue,citecolor=blue,linkcolor=blue,bookmarksopen=false,draft=false]{hyperref}

\theoremstyle{plain}
\newtheorem{theorem}{Theorem}
\newtheorem{lemma}[theorem]{Lemma}
\newtheorem{corollary}[theorem]{Corollary}

\newtheorem{definition}[theorem]{Definition}
\newtheorem{example}[theorem]{Example}
\theoremstyle{remark}
\newtheorem{remark}[theorem]{Remark}

\providecommand{\proofname}{Proof}
\renewenvironment{proof}[1][\proofname]{\par\noindent\textit{#1.}\ \ignorespaces}{\par}
\renewcommand{\qed}{\hfill$\scriptstyle\square$}

\allowdisplaybreaks

\vbadness=\maxdimen

\makeatletter
\renewcommand*{\top}{%
  {\mathpalette\@transpose{}}%
}
\newcommand*{\@transpose}[2]{%
  \scriptsize
  \raisebox{\depth}{$\m@th#1\mathsf{T}$}%
}
\makeatother

\DeclareMathOperator{\Diag}{Diag}
\DeclareMathOperator{\diag}{diag}
\DeclareMathOperator{\ldet}{ldet}

\DeclareMathOperator{\Tr}{Tr}
\DeclareMathOperator{\rank}{rank}

\renewcommand{\circeq}{\mathrel{\ooalign{\hss$\circ$\hss\cr$\equiv$}}}

\newcommand{\znaturalthing}{\mathfrak{z}_{\mbox{\protect\tiny $\mathcal{N}$}}}

\newcommand{\zdoptthing}{{\mathfrak{z}}_{\text{\normalfont\protect\tiny  D-Opt}}}

\newcommand{\zmespthing}{z_{\mbox{\normalfont\protect\tiny MESP}}}

\newcommand{\znlpthingmesp}{{z}_{\text{\normalfont\protect\tiny NLP-Id$(\gamma)$}}}

\newcommand{\DOPT}{
\mbox{\rm D-Opt}}
\newcommand{\MESP}{
\mbox{\rm MESP}}

\newcommand{\NLPIdsym}{
\mbox{\rm{NLP-Id}}}

\newcommand{\NLPId}{\hyperlink{NLPIdtarget}{\NLPIdsym}}

\begin{document}

\begin{frontmatter}

\title{New insights into the NLP-Id bound for maximum-entropy sampling}

\author[iowa]{Kurt Anstreicher}
\ead{kurt-anstreicher@uiowa.edu}

\author[ufrj]{Marcia Fampa}
\ead{fampa@cos.ufrj.br}

\author[umich]{Jon Lee}
\ead{jonxlee@umich.edu}

\author[cuhk]{Yongchun Li}
\ead{yongchunli@cuhk.edu.cn}

\author[umich]{Gabriel Ponte}
\ead{gabponte@umich.edu}

\address[iowa]{University of Iowa}
\address[ufrj]{Federal University of Rio de Janeiro}
\address[umich]{University of Michigan}
\address[cuhk]{The Chinese University of Hong Kong, Shenzhen}

\begin{abstract}
We establish new properties of the NLP-Id upper bound for the maxi\-mum-entropy sampling problem (MESP).
In particular, we give a detailed look at the concavity of its objective function as a function of the scaling parameter employed for NLP bounds for MESP.
This leads to more relaxed choices for the scaling parameter for NLP-Id and
even improved upper bounds for MESP.
\end{abstract}

\end{frontmatter}

\clubpenalty = 0
\widowpenalty = 0
\displaywidowpenalty = 0
\sloppy

\section{Introduction}\label{sec:int}

The so-called ``NLP bound'' for the maximum-entropy sampling problem (MESP) was introduced over thirty years ago (see \cite{AFLW_IPCO,AFLW_Using}). It was the first convex-relaxation based upper bound for MESP. The concavity of its objective function and the quality of the bound depend on the choice of certain parameters. 
The simplest choice, originally suggested in \cite{AFLW_Using} is the 
so-called ``Identity strategy'', yielding what is known as the NLP-Id bound. 
This manifestation of the NLP bound has received renewed attention recently.
\cite{ponte2026relationshipmesp01dopt} established that the NLP-Id bound for any instance of MESP is equivalent to the so-called ``natural bound'' on a closely related instance of the 0/1 D-optimality (0/1 D-Opt) problem. Leveraging this, in what follows, we gain a new interpretation of the ``complementary NLP-Id'' bound. Further, by carefully analyzing the Hessian of the objective function, we are able to gain a fuller understanding of the concavity of the objective function
of the NLP-Id relaxation as a function of the scaling parameter.

Other convex-programming based bounds for MESP also involve a positive scaling parameter, notably the ``linx bound'' and the ``BQP bound''. The quality of these bounds depends critically on its choice, as is also true for the NLP bounds. Furthermore, the associated relaxations, ``linx'' and ``BQP'', are convex programs for all positive scaling parameters, and the 
bounds themselves are convex in the logarithm of the scaling parameter (see \cite{Mixing}).
In contrast, the family of NLP bounds 
does not have these nice properties, and 
so a strong motivation of our work is 
to understand the behavior of NLP relaxations as a function of the scaling parameter.

\smallskip
\noindent{\bf Organization and contributions.} 
In \S\ref{sec:equivalence},
we carefully define MESP and the NLP-Id bound, and 0/1 D-Opt and the natural bound, 
and extend some connections established in \cite{ponte2026relationshipmesp01dopt}. In \S\ref{sec:results}, we use these connections to establish a new interval of concavity of the NLP-Id objective function, for its scaling parameter. We also interpret the new interval of concavity in terms of the ``complementary NLP-Id bound'', showing for the first time that there is a connection between the scaling and complementation operations.
The analysis in \S\ref{sec:results} does not consider the Hessian of the objective function of NLP-Id, but 
in \S\ref{sec:hessian}, we analyze the Hessian to directly verify the intervals of concavity from \S\ref{sec:results}.
Then, in \S\ref{sec:forbidden}, we use the Hessian to give a checkable sufficient condition for non-concavity of the objective function (at some given $\gamma$),
and we demonstrate \emph{numerically} that 
the objective function can be concave at values outside of the range established in 
\S\ref{sec:results}. Further, for such values of $\gamma$, we can have improved upper bounds on the optimal value of MESP.

\smallskip
\noindent{\bf Notation.}
We denote any all-zero vector or matrix by $0$.
We denote any all-one vector
by $\mathbf{e}$,
 and the identity matrix of order $n$ by $I_n$\,.
 We let $\mathbb{S}^n$  (resp., $\mathbb{S}^n_+$~, $\mathbb{S}^n_{++}$)
 denote the set of symmetric (resp., positive-semidefinite, positive-definite) matrices of order $n$.
 We use the usual symbols $\prec,\preceq,\succ,\succeq$ in the context of the L\"owner ordering.
 We let $\Diag(x)$ denote the $n\times n$ diagonal matrix with diagonal elements given by the components of $x\in \mathbb{R}^n$, we let $\diag(X)$ denote the $n$-vector with elements given by the diagonal elements of $X\in\mathbb{R}^{n\times n}$,
and we let $\Diag(X):=\Diag(\diag(X))$, for $X\in\mathbb{R}^{n\times n}$.
  For matrix $X$,  $X_{S,T}$\, is the submatrix  with row (column) indices $S$ ($T$);
 we write $i$ for $\{i\}$.
We let
$\ldet$ denote the natural logarithm of the determinant. 
For matrices of the same shape,
$X_1\circ X_2$ is the Hadamard (i.e., element-wise) product, and $X^{\scriptscriptstyle (2)}:=X\circ X$.
We let $N:=\{1,2,\ldots,n\}$.  
For $X\in\mathbb{S}^n$ and $i\in N$, we let $\lambda_i(X)$ denote the $i$-th greatest eigenvalue of $X$,
and we let  $\lambda(X):=(\lambda_1(X),\lambda_2(X),\ldots,\allowbreak \lambda_n(X))^\top$.
For extra clarity, we define $\lambda_{\max}(X):=\lambda_1(X)$
and $\lambda_{\min}(X):=\lambda_n(X)$.
Specifically for $C \in \mathbb{S}^n_{+}$ (which plays a special role for us), we define $\lambda_{\max} := \lambda_{\max}(C)$,  $\lambda_{\min} := \lambda_{\min}(C)$, and  $\mu_{\max}$ ($\mu_{\min}$) as the multiplicity of $\lambda_{\max}$ ($\lambda_{\min}$) as an eigenvalue of $C$. 


\section{MESP, 0/1 D-Opt, and some connections}\label{sec:equivalence}

\noindent{\bf MESP.} 
Let $C$ be a symmetric positive-semidefinite matrix with rows/columns
indexed from $N$.
For $0< s < n$,
we define the \emph{maximum-entropy sampling problem}
\begin{align*}
\tag{MESP$(C,s)$}\label{MESP}
&\zmespthing(C,s):= \\
&\quad \max \left\{ \ldet\left( C_{S(x),S(x)}\right)\right.\\
&\qquad\left.:~ \mathbf{e}^\top x =s,~ x\in\{0,1\}^n\right\},
\end{align*}
where 
 $S(x)$ denotes the support of $x$. 
 For feasibility, we assume that $r:=\rank(C)\geq s$. 
 \ref{MESP} was introduced by \cite{SW}; also see \cite{FLbook} and the many references therein. Briefly, in the Gaussian case,  $\ldet (C_{S,S})$ is
 proportional to the ``differential entropy'' (see \cite{Shannon}) of a vector of random variables 
 having covariance matrix $C_{S,S}$\,.  So \ref{MESP} seeks to
 find the ``most informative'' $s$-subvector from an $n$-vector following a joint 
 Gaussian distribution. \ref{MESP} finds application in many areas, e.g. environmental monitoring (see \cite[Chap. 4]{FLbook}). 

 \medskip
\noindent{\bf 0/1 D-Opt.}  
 Consider the 0/1 $\DOPT$ problem formulated as 
\begin{align*}\label{DOPT}\tag{D-Opt$(A,B,s)$}
\textstyle
&\zdoptthing(A,B,s):=\\
&\quad \max\!\left\{
\ldet\! \left( A^\top \Diag(x) A \!+\! B^\top B\right)
\right.\\
&\left.\qquad:~ \mathbf{e}^\top x = s,\, x\!\in\!\{0,1\}^n
\right\},
\end{align*}
where $A \in \mathbb{R}^{n\times m}$  (with no all-zero rows), $B \in \mathbb{R}^{q\times m}$, $q\geq 1$  (with all-zero rows allowed), 
{\scriptsize$\begin{pmatrix}A\\ B \end{pmatrix}$} has full column rank $m$, and $n>s\geq m-\rank(B)$. In the world of statistics, rows of {\scriptsize$\begin{pmatrix}A\\ B \end{pmatrix}$}
are considered to be ``design points'' encoding levels 
of $m$ factors. We are required to choose all rows of $B$, and we wish to choose, additionally, $s$ distinct rows from $A$. 
Here, $B^\top B$ is known as the ``existing Fisher-Information matrix''; if $B$ has full column rank, then 
\ref{DOPT} is called a \emph{0/1 D-optimal data fusion instance.}
The goal of \ref{DOPT} is to minimize the generalized variance for parameter estimates in a least-squares model, based on the chosen design points. 
The assumption that {\scriptsize$\begin{pmatrix}A\\ B \end{pmatrix}$} has full column rank
implies that every feasible solution of \ref{DOPT} yields an ``identifiable'' regression model
(i.e., the associated least-squares problem has a unique solution). 
The 0/1 $\DOPT$ problem and its variants are truly fundamental in the design of experiments; see \cite{puk} and \cite{PonteFampaLeeMPB}.

Following \cite{ponte2026relationshipmesp01dopt}, for
combinatorial-optimization problems
$P$ and $Q$
for which the feasible sets
are identical point-sets in $\{0,1\}^n$
and a constant $K\in\mathbb{R}$,
such that every feasible solution has 
objective value in $Q$ exactly $K$ more than in $P$, we write $Q\equiv P+ K$.
If instead, the feasible points of $Q$ are precisely the feasible points of $P$ under the complementation involution
($\hat{x}\mapsto\mathbf{e}-\hat{x}$), and for every feasible solution $\hat{x}$ of $P$, the objective value of $Q$ at $\mathbf{e}-\hat{x}$ is exactly $K$ more than the objective value of $P$ at $\hat{x}$, we write $Q \circeq P + K$.

\cite{AFLW_Using} first pointed out 
a scaling principle for $\MESP$:
that is, for $\gamma>0$, we have  $\MESP(C,s)\equiv\MESP(\gamma C,s) -s \log \gamma$,
and so $\zmespthing(C,s)= \zmespthing(\gamma C,s) -s \log \gamma$. We note that any upper bound for $\MESP(\allowbreak \gamma C,s)$ minus $s \log \gamma$ is also an upper bound for $\MESP(C,s)$, which we call a \emph{scaled upper bound}.
Also, \cite{AFLW_Using} pointed out  
a complementing principle for $\MESP$:
if $C$ is positive definite, 
then $\MESP(C,s)\circeq\MESP(C^{\scriptscriptstyle -1},n-s) + \ldet (C)$, here employing the complementation involution,
and so $\zmespthing(C,s)\allowbreak =\allowbreak \zmespthing(C^{\scriptscriptstyle -1},n-s) + \ldet (C)$.
In what follows, we say that 
$\MESP(C,s)$ and $\MESP(C^{\scriptscriptstyle -1},n-s)$ are \emph{complementary $\MESP$ instances}.  We note that any upper bound for $\MESP(C^{\scriptscriptstyle -1},n-s)$ plus $\ldet (C)$ is also an upper bound for $\MESP(C,s)$, which we call a \emph{complementary upper bound}.

We begin by introducing parametrically scaled versions of two maps described in \cite{ponte2026relationshipmesp01dopt}.

\begin{definition} [parametrically scaled maps $\mathcal{D}_\gamma$ and $\mathcal{F}_\gamma$]\label{def:newmapping}~ 
\begin{itemize}
\item 
Consider an arbitrary $\MESP$ instance $\MESP(C_{n\times n}\,,s)$ 
and a real Schur decomposition $\Phi\Lambda\Phi^\top$ of $C$. Let $\gamma\in(0,1/\lambda_{\max}]$. Define $\mathcal{D}_\gamma(\MESP(C,s);\Phi,\gamma):=\DOPT(A,B,n-s)$,
where
$A_{n\times n}:=\Phi(I_n-\gamma\Lambda)^{\scriptscriptstyle 1/2}$, and $B_{n\times n}:=(\gamma\Lambda)^{\scriptscriptstyle 1/2}$. 
Note that $\rank(B)=\rank(C)$, and $\rank(A)\in\{n-\mu_{\max}\,,n\}$\,. Additionally, $A^\top A+B^\top B=I_n$ and $AA^\top=I_n - \gamma C$.

\item 
Consider a $\MESP$ instance $\MESP(C_{n\times n}\,,s)$, with $C\in\mathbb{S}^n_{++}$\,. Let $\gamma\in[1/\lambda_{\min}\,,\allowbreak +\infty)$ and 
$A\in\mathbb{R}^{n\times m}$ be such that $AA^\top:=\gamma C-I_n$\,. 
Define $\mathcal{F}_{\gamma}(\MESP(C,s);\allowbreak A,\gamma):=\DOPT(A,I_m\,,s)$, which is a 0/1 D-optimal data-fusion instance.  Note that $m\geq \rank(A)\in\{n-\mu_{\min}\,,n\}$\,.
\end{itemize}
\end{definition}
These maps take instances of 
\ref{MESP} to instances of 0/1 D-Opt. The map 
$\mathcal{D}_\gamma$ complements the variables, while the map $\mathcal{F}_\gamma$ does not. For $\gamma=
1/\lambda_{\max}$
(resp., $1/\lambda_{\min}$), the map $\mathcal{D}_\gamma$
(resp., $\mathcal{F}_\gamma$)
becomes the map $\mathcal{D}$ (resp., $\mathcal{F}$) from \cite{ponte2026relationshipmesp01dopt}.

\begin{theorem}\label{thm:equivDF}
\phantom{.}
\begin{itemize}
\item {\rm(adapted from 
\cite[Thm. 10]{ponte2026relationshipmesp01dopt})}
For an arbitrary $\MESP$ instance $\MESP(C,s)$, $\gamma\in(0,1/\lambda_{\max}]$, and 
a real Schur decomposition $\Phi\Lambda\Phi^\top$ of $C$, we have 
$\MESP(C,s)\circeq\mathcal{D}_{\gamma}(\MESP(C,s);\Phi,\gamma)-s\log \gamma$.
\item {\rm(adapted from 
\cite[Thm. 1]{li2022d})} For a positive-definite $\MESP$ instance $\MESP(C,s)$,  $\gamma\in[1/\lambda_{\min}\,,+\infty)$, and an 
$A\in\mathbb{R}^{n\times m}$ such that $AA^\top:=\gamma C-I_n$\,, we have 
$\MESP(C,s)\equiv\mathcal{F}_{\gamma}(\MESP(C,s);\allowbreak A,\gamma)-s\log \gamma$.
\end{itemize}
\end{theorem}

Next, we define a mild generalization of the NLP-Id upper bound for \ref{MESP} proposed in \cite{AFLW_Using}.

\medskip

\noindent \textbf{NLP-Id bound.} For $\gamma>0$, the \emph{(scaled)  NLP-Id bound} for \ref{MESP} is 
\begin{align}\label{nlpid_bound}\tag{NLP-Id$(\gamma)$}
&\hypertarget{znlptargetmesp}{\znlpthingmesp}(C,s) := \\[-4pt]
&\quad\max \left\{\vphantom{\Big)}\ldet\big(I_n +  \Diag(x)^{\scriptscriptstyle 1/2}(\gamma C \!-\!I_n)\Diag(x)^{\scriptscriptstyle 1/2} \big)\right. \nonumber\\[-4pt] 
&\qquad \left. - s\log \gamma ~:~
  \mathbf{e}^\top x\!=\!s,~ 
x\in[0,1]^n\vphantom{\Big)}\right\}.\nonumber 
\end{align}
The ordinary NLP-Id bound of \cite{AFLW_Using} has $\gamma:=1/\lambda_{\max}$\,.

The natural bound for \ref{DOPT} was proposed in \cite{Welch}; also see \cite{PonteFampaLeeMPB}.

\medskip

\noindent\textbf{Natural bound.} 
The \emph{natural bound} for \ref{DOPT} is 
\begin{minipage}{\linewidth}\begin{align*}\label{natural_bound}\tag{\mbox{$\mathcal{N}$}} 
&\textstyle \hypertarget{znaturaltarget} 
\textstyle \hypertarget{znaturaltarget}{\znaturalthing}(A,B,s):= \\[-4pt]
&\quad \max\! \left\{ \ldet \left(A^\top \Diag(x) A + B^\top B \right) 
\right.\\[-4pt]
&\qquad\left.
~:~ \mathbf{e}^\top x\!=\!s,~ 
x\in[0,1]^n
\right\}.
\end{align*}
\end{minipage}


Generalizing from \cite{ponte2026relationshipmesp01dopt}, we have the following definitions of natural bounds induced by our maps. These bounds are upper bounds for $\zmespthing$\,, due to Theorem~\ref{thm:equivDF}.

\begin{definition}
[$\mathcal{D}_\gamma$-induced and $\mathcal{F}_\gamma$-induced natural bounds for $\MESP(C,s)$]
\phantom{.}
\begin{itemize}
\item Considering a real Schur decomposition $\Phi\Lambda\Phi^\top$ of $C$ and $\gamma\in(0,1/\lambda_{\max}]$, the $\mathcal{D}_\gamma$\emph{-induced natural bound} for $\MESP(C,s)$ is defined as  $\znaturalthing(A,B,n-s) - s\log \gamma$, where $\DOPT(A,B,n-s):=\mathcal{D}_\gamma(\MESP(C,s);\Phi,\gamma)$.
\item Considering $A\in\mathbb{R}^{n\times m}$  such that $AA^\top:=\gamma C-I_n$\,, where $C$ is positive-definite,   and $\gamma\in[1/\lambda_{\min}\,,+\infty)$, the $\mathcal{F}_\gamma$-\emph{induced natural bound} for $\MESP(C,s)$ is defined as  $\znaturalthing(A,I_m\,,s) - s\log \gamma$, where $\DOPT(A,I_m\,,s):=\mathcal{F}_\gamma(\MESP(C,s);A,\gamma)$.
\end{itemize}
\end{definition}


\section{Another interval of $\gamma$ for concavity of \ref{nlpid_bound}}\label{sec:results}

The following result was established in \cite{ponte2026relationshipmesp01dopt} for $\gamma:=1/\lambda_{\max}$\,. We extend the result for all $\gamma\in(0,1/\lambda_{\max}]$. Although the proof from \cite{ponte2026relationshipmesp01dopt} is directly extended, we include it here for completeness.

\begin{lemma}[\protect{see \cite[Lem.~27]{ponte2026relationshipmesp01dopt}}]\label{lem:nlp-natural-equiv-D-every-x}
    Let  $C \!\in \!\mathbb{S}^n_{+}$   with greatest eigenvalue $\lambda_{\max}$  and a real Schur decomposition $\Phi\Lambda\Phi^\top$. Let $\gamma\!\in\!(0,1/\lambda_{\max}]$,  $A:=\Phi(I_n-\gamma\Lambda)^{\scriptscriptstyle 1/2}$ and $B:=(\gamma\Lambda)^{\scriptscriptstyle 1/2}$. Then, we have
\begin{align*}
\ldet\!\left( I_n + \Diag(\hat{x})^{\scriptscriptstyle 1/2}(\gamma C\!-\!I_n)\Diag(\hat{x})^{\scriptscriptstyle 1/2} \right) \\= \ldet\!\left( {A}^\top \Diag(\mathbf{e}\!-\!\hat{x}) {A} \!+\! {B}^\top {B}\right),
\end{align*}
for all $0\leq \hat{x}\leq \mathbf{e}$ such that  the $\ldet$ functions above have a finite value. 
\end{lemma}
\begin{proof}
   We have  
    \begin{align*}
        &\ldet\left( I_n + \Diag(\hat{x})^{\scriptscriptstyle 1/2}(\gamma C-I_n)\Diag(\hat{x})^{\scriptscriptstyle 1/2} \right) \\
        &\quad  = \ldet\left( I_n - \Diag(\hat{x})^{\scriptscriptstyle 1/2}{A}{A}^{\top}\Diag(\hat{x})^{\scriptscriptstyle 1/2}\right)\\
        &\quad =\ldet\left( I_n - {A}^\top\Diag(\hat{x}){A}\right)\\
        &\quad = \ldet\left( {A}^\top {A} +B^\top B   - {A}^\top \Diag(\hat{x}) {A}\right)\\
    &\quad =\ldet\left( {A}^\top \Diag(\mathbf{e}-\hat{x}) {A} + {B}^\top {B}\right),
    \end{align*}
where  we use that ${A}^\top {A} +B^\top B =I_n$\,.
\qed \end{proof}

Similarly, we have the following result.

\begin{lemma}\label{lem:nlp-natural-equiv-F-every-x}
    Let  $C \!\in \!\mathbb{S}^n_{++}$  with  least eigenvalue $\lambda_{\min}$\,. Let $\gamma\in[1/\lambda_{\min}\,,+\infty)$, and let $A\in\mathbb{R}^{n\times m}$ be such that $AA^\top:=\gamma C-I_n$\,. Then, we have 
\begin{align*}
\ldet\!\left( I_n + \Diag(\hat{x})^{\scriptscriptstyle 1/2}(\gamma C\!-\!I_n)\Diag(\hat{x})^{\scriptscriptstyle 1/2} \right) \\= \ldet\!\left( {A}^\top \Diag(\hat{x}) {A} \!+\! I_m\right),
\end{align*}
for all $0\leq \hat{x}\leq \mathbf{e}$ such that  the $\ldet$ functions above have a finite value. 
\end{lemma}
\begin{proof}
   We have  
    \begin{align*}
        &\ldet\left( I_n + \Diag(\hat{x})^{\scriptscriptstyle 1/2}(\gamma C-I_n)\Diag(\hat{x})^{\scriptscriptstyle 1/2} \right) \\
        &\quad  = \ldet\left( I_n + \Diag(\hat{x})^{\scriptscriptstyle 1/2}{A}{A}^{\top}\Diag(\hat{x})^{\scriptscriptstyle 1/2}\right)\\
        &\quad =\ldet\left( I_m + {A}^\top\Diag(\hat{x}){A}\right).
    \end{align*}
\qed \end{proof}

\begin{theorem}\label{thm:2intervals}
Let  $C \!\in \!\mathbb{S}^n_{+}$\, with greatest eigenvalue $\lambda_{\max}$\,. For all $\gamma\in
(0,1/\lambda_{\max}]
$,  
{\rm\ref{nlpid_bound}}
is a convex optimization problem.
Furthermore, if $C \!\in \!\mathbb{S}^n_{++}$\, with least eigenvalue $\lambda_{\min}$\,, then, for all $\gamma\in[1/\lambda_{\min}\,,+\infty)$,  
{\rm\ref{nlpid_bound}}
is a convex optimization problem.
\end{theorem}

\begin{proof}
    The result follows directly from Lemmas \ref{lem:nlp-natural-equiv-D-every-x} and \ref{lem:nlp-natural-equiv-F-every-x}.
    \qed
\end{proof}

The fact that {\rm{\ref{nlpid_bound}}} is 
a convex optimization problem for $\gamma=1/\lambda_{\max}$ is a well-known result from \cite{AFLW_Using}, and that it is actually convex on all of $(0,1/\lambda_{\max}]$ can be derived from further results in \cite{AFLW_Using}. The second half of Theorem \ref{thm:2intervals}, that 
{\rm{\ref{nlpid_bound}}} is 
a convex optimization problem for an additional interval, is entirely new and surprising.

The following result was established in \cite{ponte2026relationshipmesp01dopt} for $\gamma:=1/\lambda_{\max}$\,. We extend the result for all $\gamma\in(0,1/\lambda_{\max}]$. Although the proof from \cite{ponte2026relationshipmesp01dopt} is directly extended, we include it for completeness.

\begin{theorem}[\protect{\cite[Thm.~28]{ponte2026relationshipmesp01dopt}}]\label{thm:nlp-natural-D-equiv}
    For  $C \!\in \!\mathbb{S}^n_{+}$\, with greatest eigenvalue $\lambda_{\max}$\,,  consider the $\MESP$ instance $\MESP(C,s)$ and a real Schur decomposition $\Phi\Lambda\Phi^\top$ of $C$. Let $\gamma\in(0,1/\lambda_{\max}]$, and let $\DOPT(A,B,\allowbreak n-s):=\mathcal{D}_\gamma(\MESP(C,s);\Phi,\gamma)$. Then, 
the $\mathcal{D}_{\gamma}$-induced natural bound is equal to the 
{\rm\ref{nlpid_bound}}
bound for $\MESP(C,s)$; that is, 
$\znaturalthing(A,B,n-s) - s\log \gamma=
\znlpthingmesp(C,s)$.
\end{theorem}

\begin{proof}  
   By the definition of $\mathcal{D}_\gamma$\,, we have 
$A:=\Phi(I_n-\gamma\Lambda)^{\scriptscriptstyle 1/2}$ and $B:=(\gamma\Lambda)^{\scriptscriptstyle 1/2}$.   Let $\hat{x}$ be a feasible solution to 
{\rm\ref{nlpid_bound}}
with finite objective value. Then,  from Lemma~\ref{lem:nlp-natural-equiv-D-every-x}, we have  
\begin{equation*} 
\begin{array}{l}
        \ldet\left( I_n + \Diag(\hat{x})^{\scriptscriptstyle 1/2}(\gamma C-I_n)\Diag(\hat{x})^{\scriptscriptstyle 1/2} \right) - s\log\gamma\\
     \quad =\ldet\left( {A}^\top \Diag(\mathbf{e}-\hat{x}) {A} + {B}^\top {B}\right) -s\log\gamma,
    \end{array}
    \end{equation*}
where on the left-hand side of the equation  we have the objective value of $\NLPId(\gamma)$ at $\hat{x}$, and on the  right-hand side we have the objective value of \ref{natural_bound} at the corresponding feasible solution $\mathbf{e}\!-\!\hat{x}$, subtracted by  $s\!\log\!\gamma$.
The result follows.
\qed \end{proof}

Next, we derive a similar result when $\gamma$ belongs to the new interval of convexity.

\begin{theorem}\label{thm:nlp-natural-F-equiv}
    For  $C \!\in \!\mathbb{S}^n_{++}$\, with least eigenvalue $\lambda_{\min}$\,,  consider the $\MESP$ instance $\MESP(C,s)$. Let $\gamma\in[1/\lambda_{\min}\,,+\infty)$, let $A\in\mathbb{R}^{n\times m}$ be such that $AA^\top:=\gamma C-I_n$\,, and  let $\DOPT(A,I_m\,,s):=\mathcal{F}_{\gamma}(\MESP(C,s);A,\gamma)$. Then, 
the $\mathcal{F}_{\gamma}$-induced natural bound is equal to the 
{\rm\ref{nlpid_bound}}
bound for $\MESP(C,s)$; that is, 
$\znaturalthing(A,I_m,s) - s\log \gamma=
\znlpthingmesp
(C,s)$.
\end{theorem}

\begin{proof}  
  Let $\hat{x}$ be a feasible solution to 
  {\rm\ref{nlpid_bound}}
  with finite objective value. Then,  from Lemma \ref{lem:nlp-natural-equiv-F-every-x}, we have  
\begin{equation*} 
\begin{array}{l}
       \ldet\left( I_n + \Diag(\hat{x})^{\scriptscriptstyle 1/2}(\gamma C-I_n)\Diag(\hat{x})^{\scriptscriptstyle 1/2} \right) - s\log\gamma\\ 
     \quad =\ldet\left( {A}^\top \Diag(\hat{x}) {A} + I_m\right) -s\log\gamma,
    \end{array}
    \end{equation*}
where on the left-hand side of the equation we have the objective value of 
{\rm\ref{nlpid_bound}}
at $\hat{x}$,  and on the  right-hand side we have the objective value of \ref{natural_bound} at  $\!\hat{x}$, shifted by  $s\!\log\!\gamma$.
The result follows.
\qed \end{proof}

Next, we provide a new interpretation of the well-known complementary \ref{nlpid_bound} bound for MESP. By considering the new convexity interval for $\gamma$, we show that it is equivalent to the original \ref{nlpid_bound} bound
for \ref{MESP} when $C$ is positive-definite.

\begin{theorem}\label{eqcomp-orig}
    Let $C\in\mathbb{S}^n_{++}$\,.  We have $\hat\gamma\in(0,1/\lambda_{\max}(C^{\scriptscriptstyle -1})]$ if and only if  $1/\hat\gamma \in[1/\lambda_{\min}(C),+\infty)$. 
    Furthermore, the complementary 
    {\rm\ref{nlpid_bound}}
    bound for $\gamma:=\hat\gamma$  is equal to  the   
    {\rm\ref{nlpid_bound}}
    bound for $\gamma:=1/\hat\gamma$; that is,
    \begin{align}
    \label{eq:znlpid-diff-comp}
 & 
{z}_{\text{\normalfont\protect\tiny NLP-Id$(\hat\gamma)$}}
 (C^{\scriptscriptstyle -1},n-s) +\ldet(C) =
{z}_{\text{\normalfont\protect\tiny NLP-Id$(1/\hat\gamma)$}}
(C,s).\nonumber
\end{align}
\end{theorem}
\begin{proof}
    The if and only if statement is straightforward. For the ``furthermore part'',
    consider a real Schur decomposition $\Phi\Lambda\Phi^\top$ of $C$. Let $0\leq \hat{x}\leq \mathbf{e}$, such that $\mathbf{e}^\top\hat x=s$ and the $\ldet$ functions in the following have a finite value. 
    
    Let $\tilde A:= \Phi(I_n-\hat\gamma \Lambda^{\scriptscriptstyle -1})^{\scriptscriptstyle 1/2}$ and $\tilde B:= (\hat\gamma \Lambda^{\scriptscriptstyle -1})^{\scriptscriptstyle 1/2}$. Then,  we have
\begin{equation*} 
\begin{array}{l}
        \ldet\left( I_n + \Diag(\mathbf{e}-\hat{x})^{\scriptscriptstyle 1/2}(\hat \gamma C^{\scriptscriptstyle -1}-I_n)\Diag(\mathbf{e}-\hat{x})^{\scriptscriptstyle 1/2} \right)\\
     \qquad - (n\!-\!s)\log\hat\gamma +\ldet(C)\\
     \quad =\ldet\left( {\tilde A}^\top \Diag(\hat{x}) {\tilde A} + {\tilde B}^\top {\tilde B}\right) -(n\!-\!s)\log\hat\gamma +\ldet(C)\\
     \quad =\ldet\left((I_n\!-\!\hat\gamma \Lambda^{\scriptscriptstyle -1})^{\scriptscriptstyle 1/2} \Phi^\top \Diag(\hat{x}) \Phi(I_n\!-\!\hat\gamma \Lambda^{\scriptscriptstyle -1})^{\scriptscriptstyle 1/2} + \hat\gamma\Lambda^{\scriptscriptstyle -1}\right)\\
     \qquad -(n\!-\!s)\log\hat\gamma +\ldet(C),
    \end{array}
    \end{equation*}
     where the first equation follows from Lemma \ref{lem:nlp-natural-equiv-D-every-x}.
     
    Let $\bar A=\Phi(\frac{1}{\hat\gamma}\Lambda-I_n)^{\scriptscriptstyle 1/2}$.  Then, $\bar{A}\bar{A}^\top=\frac{1}{\hat\gamma} C-I_n$\,, and 
    \begin{equation*} 
\begin{array}{l}
       \ldet\left( I_n + \Diag(\hat{x})^{\scriptscriptstyle 1/2}(\frac{1}{\hat\gamma} C-I_n)\Diag(\hat{x})^{\scriptscriptstyle 1/2} \right) - s\log\frac{1}{\hat\gamma}\\ 
     \quad =\ldet\left( {\bar A}^\top \Diag(\hat{x}) {\bar A} + I_m\right) +s\log{\hat\gamma}\,
     \\ 
     \quad =\ldet\left((\frac{1}{\hat\gamma}\Lambda-I_n)^{\scriptscriptstyle 1/2}\Phi^\top \Diag(\hat{x}) \Phi(\frac{1}{\hat\gamma}\Lambda-I_n)^{\scriptscriptstyle 1/2} + I_m\right)\\
     \qquad +s\log{\hat\gamma}\,
     \\ 
     \quad =\ldet\left((\hat\gamma \Lambda^{\scriptscriptstyle -1})^{\scriptscriptstyle 1/2}(\frac{1}{\hat\gamma}\Lambda\!-\!I_n)^{\scriptscriptstyle 1/2}\Phi^\top \Diag(\hat{x}) \Phi(\frac{1}{\hat\gamma}\Lambda\!-\!I_n)^{\scriptscriptstyle 1/2} (\hat\gamma \Lambda^{\scriptscriptstyle -1})^{\scriptscriptstyle 1/2}\right.\\
     \qquad\left. + \hat\gamma\Lambda^{\scriptscriptstyle -1}\right)  +\ldet(\frac{1}{\hat\gamma}\Lambda)+s\log{\hat\gamma}\,
     \\ 
     \quad =\ldet\left(I_n\!-\!\hat\gamma \Lambda^{\scriptscriptstyle -1})^{\scriptscriptstyle 1/2}\Phi^\top \Diag(\hat{x}) \Phi(I_n\!-\!\hat\gamma \Lambda^{\scriptscriptstyle -1})^{\scriptscriptstyle 1/2} + \hat\gamma\Lambda^{\scriptscriptstyle -1}\right)\\
     \qquad +\ldet(C)-(n\!-\!s)\log{\hat\gamma},
    \end{array}
    \end{equation*}
    where the first equation follows from Lemma \ref{lem:nlp-natural-equiv-F-every-x}.
    The result follows. \qed
\end{proof}

From the result in \cite[Lem. 23]{ponte2026relationshipmesp01dopt}, we can conclude that the \ref{nlpid_bound} bound is non-increasing in $\gamma$ on $(0,1/\lambda_{\max}]$. In the following, we use the same argument to analyze the monotonicity of the \ref{nlpid_bound} in the other interval where the convexity of the relaxation is guaranteed. 

\begin{lemma}\label{lem:natural_to_nlp_modified}
Let $C\in\mathbb{S}^n_{++}$ and $0<s<n$. The {\rm\ref{nlpid_bound}}  bound is non-increasing in $\gamma$ on $(0,1/\lambda_{\max}]$ and  non-decreasing in $\gamma$ on $[1/\lambda_{\min},+\infty)$.
\end{lemma}

\begin{proof}
It suffices to establish the claim on $(0,1/\lambda_{\max}]$: the behavior on
$[1/\lambda_{\min},+\infty)$ then follows from the equivalence between the
\ref{nlpid_bound} bound and its complementary counterpart established in
Theorem~\ref{eqcomp-orig}.

Let   $0\leq \hat x\leq \mathbf{e}$, with $\mathbf{e}^\top \hat{x}=s$. Define $\Theta(\gamma):= I_n + \Diag(\hat{x})^{\scriptscriptstyle 1/2}(\gamma C-I_n)\Diag(\hat{x})^{\scriptscriptstyle 1/2}$. Assume that $\Theta(\gamma) \succ 0$ for all $\gamma \in \left(0,1/\lambda_{\max}\right]$.  Let 
     \begin{equation*} 
         h(\gamma) := \ldet( \Theta(\gamma)) - s\log\gamma,
     \end{equation*}
     which gives the objective value of \ref{nlpid_bound} at the fixed point $\hat x$, for a given $\gamma$.

We have that 
\begin{align}
    h'(\gamma) &= \gamma^{\scriptscriptstyle -1}   \Tr(\Theta(\gamma)^{\scriptscriptstyle -1}\gamma \Diag(\hat{x})^{\scriptscriptstyle 1/2} C\Diag(\hat{x})^{\scriptscriptstyle 1/2} ) - \gamma^{\scriptscriptstyle -1} s\nonumber\\
    &=\gamma^{\scriptscriptstyle -1}\Tr(\Theta(\gamma)^{\scriptscriptstyle -1}(\Theta(\gamma) - \Diag(\mathbf{e}-\hat{x}))) -  \gamma^{\scriptscriptstyle -1}\Tr(\Diag(\hat{x}))\nonumber\\
     &=\gamma^{\scriptscriptstyle -1}\Tr(I_n -\Theta(\gamma)^{\scriptscriptstyle -1}\Diag(\mathbf{e}-\hat{x})) -  \gamma^{\scriptscriptstyle -1}\Tr(\Diag(\hat{x})) \nonumber\\
    &= \gamma^{\scriptscriptstyle -1}\Tr(\Diag(\mathbf{e}-\hat{x}) - \Theta(\gamma)^{\scriptscriptstyle -1}\Diag(\mathbf{e}-\hat{x}))\nonumber\\
    &= \gamma^{\scriptscriptstyle -1}\Tr((I_n - \Theta(\gamma)^{\scriptscriptstyle -1})\Diag(\mathbf{e}-\hat{x}))\nonumber\\
     &= \gamma^{\scriptscriptstyle -1}\Tr(\Diag(\mathbf{e}-\hat{x})^{\scriptscriptstyle 1/2}(I_n-\Theta(\gamma)^{\scriptscriptstyle -1})\Diag(\mathbf{e}-\hat{x})^{\scriptscriptstyle 1/2})\nonumber\\
     &= \gamma^{\scriptscriptstyle -1}\textstyle\sum_{j \in N} (1-\hat{x}_j)  (I_n-\Theta(\gamma)^{\scriptscriptstyle -1})_{jj}\,.\label{derivativeh}
     \end{align}
As $\gamma \in (0, 1/\lambda_{\max}]$\,, then $\gamma C \preceq I_n \Rightarrow \Theta(\gamma) \preceq I_n \Rightarrow \Theta(\gamma)^{\scriptscriptstyle -1} \succeq I_n\,,$ so we conclude that $h$ is non-increasing in the interval $(0,1/\lambda_{\max}]$.

Now, consider $\hat x$ as an optimal solution to \ref{nlpid_bound} for any given $\gamma\in(0,1/\lambda_{\max}]$\,. The result follows. 
\qed 
\end{proof}


\section{Analysis of the Hessian of the objective function in \ref{nlpid_bound}}\label{sec:hessian}

Next, we calculate the Hessian of the objective function in \ref{nlpid_bound}. The Hessian from \cite[Lem. 3.5.3]{FLbook} simplifies, considering the specific parameters used in the definition of \ref{nlpid_bound}. Finally, we verify that the Hessian is negative semidefinite throughout both intervals for $\gamma$ where convexity of the relaxation is guaranteed, confirming Theorem \ref{thm:2intervals} in a different way.

\begin{lemma}\label{lem:hes_f}
	Let $R(x):=I_n +  \Diag(x)^{\scriptscriptstyle 1/2}(\gamma C-I_n)\Diag(x)^{\scriptscriptstyle 1/2}$ and $f(x):=\ldet(R(x))$, for $x \in\mathbb{R}^n_{++}$ such that $R(x) \succ 0$.  The Hessian of $f$ is
\begin{equation*}
\begin{array}{l}
\nabla^{\scriptscriptstyle 2} f(x) \!=\! \Diag(x)^{\scriptscriptstyle -1}
\left(
2\Diag\left(R(x)^{\scriptscriptstyle -1}\right) \!- I_n \!- \left(R(x)^{\scriptscriptstyle -1}\right)^{\scriptscriptstyle{(2)}}
\right)
\Diag(x)^{\scriptscriptstyle -1}.
\end{array}
\end{equation*} 
\end{lemma}

\begin{theorem}\label{thm:concave}
    Let $C\in\mathbb{S}^n_{++}$\,.  Let $R(x):=I_n +  \Diag(x)^{\scriptscriptstyle 1/2}(\gamma C-I_n)\Diag(x)^{\scriptscriptstyle 1/2}$ and $f(x):=\ldet(R(x))$, for $x \in\mathbb{R}^n_{++}$ such that $R(x) \succ 0$, then $f$ is concave on $\{ x \in [0,1]^n ~:~ R(x)\succ 0\}$ when $\gamma\in(0,1/\lambda_{\max}]$ or  $\gamma \in[1/\lambda_{\min}\,,+\infty)$.
\end{theorem}

\begin{proof}
Let $M := 2\,\Diag(R(x)^{\scriptscriptstyle -1}) - I_n - (R(x)^{\scriptscriptstyle -1})^{\scriptscriptstyle (2)}$. Then
\[
\nabla^{\scriptscriptstyle 2} f(x) = \Diag(x)^{\scriptscriptstyle -1}\, M \, \Diag(x)^{\scriptscriptstyle -1}.
\]
Note that for all
$v\in\mathbb{R}^n$, we have
$v^\top \nabla^{\scriptscriptstyle 2} f(x)\, v = w^\top M w$, where  $w := \Diag(x)^{\scriptscriptstyle -1} v$,
and $v\mapsto w$ is a bijection of $\mathbb{R}^n$. So  it suffices to show that $M \preceq 0$. 
For $S\in \mathbb{S}^n$, we have
\begin{align*}
(S\!-\!I_n)\!\circ\!(S\!-\!I_n) &= S\!\circ\! S \!-\! 2(S\!\circ\! I_n) + I_n\!\circ\! I_n 
= S^{\scriptscriptstyle (2)} \!-\! 2\,\Diag(S) \!+\! I_n\,.
\end{align*}
Hence
$
M = -(R(x)^{\scriptscriptstyle -1}-I_n)^{\scriptscriptstyle (2)}.
$
Therefore, it suffices to show that either $R(x)^{\scriptscriptstyle -1}-I_n\succeq 0$ or $R(x)^{\scriptscriptstyle -1}-I_n\preceq 0$. If
\begin{align*}
&\gamma \!\in\! (0,1/\lambda_{\max}]\Rightarrow
\Diag(x)^{\scriptscriptstyle 1/2}(I_n \! -\!\gamma C)\Diag(x)^{\scriptscriptstyle 1/2}\succeq 0.\\
&\gamma \!\in\! [1/\lambda_{\min}\,, +\infty) \Rightarrow \Diag(x)^{\scriptscriptstyle 1/2}(I_n - \gamma C)\Diag(x)^{\scriptscriptstyle 1/2}\preceq 0.
\end{align*}

As $R(x)$ is an affine function of $\Diag(x)^{\scriptscriptstyle 1/2}(I_n - \gamma C)\Diag(x)^{\scriptscriptstyle 1/2}$, $R(x)$ and $\Diag(x)^{\scriptscriptstyle 1/2}(I_n - \gamma C)\Diag(x)^{\scriptscriptstyle 1/2}$ commute. Then, $R(x)^{\scriptscriptstyle -1}$ and $\Diag(x)^{\scriptscriptstyle 1/2}(I_n - \gamma C)\Diag(x)^{\scriptscriptstyle 1/2}$ commute as well, and both matrices are simultaneously 
diagonalizable: there is an
orthonormal basis $\{v_i\}$ and reals $\tau_i$ such that, for all $i$,
\begin{align*}
&\Diag(x)^{\scriptscriptstyle 1/2}(I_n - \gamma C)\Diag(x)^{\scriptscriptstyle 1/2}\, v_i = \tau_i v_i\,, \\
&~R(x)\, v_i = (1-\tau_i) v_i\,, \\
&~ R(x)^{\scriptscriptstyle -1}\, v_i = \textstyle\frac{1}{1-\tau_i}
v_i\,.
\end{align*}
Then, for all $i$, 
\[
\textstyle(R(x)^{\scriptscriptstyle -1} - I_n)v_i = \Big(\frac{1}{1-\tau_i} - 1\Big) v_i = \frac{\tau_i}{1-\tau_i}\, v_i\,.
\]
As $R(x)\succ 0$, we have $1-\tau_i > 0$ for all $i$. Consequently,
 the sign of $\frac{\tau_i}{1-\tau_i}$ equals the sign of $\tau_i$\,.
Therefore:
\begin{align*}
&\Diag(x)^{\scriptscriptstyle 1/2}(I_n \!-\! \gamma C)\Diag(x)^{\scriptscriptstyle 1/2} \succeq 0 \ (\text{all } \tau_i \ge 0) 
\Rightarrow R(x)^{\scriptscriptstyle -1} \!-\! I_n \succeq 0,\\
&\Diag(x)^{\scriptscriptstyle 1/2}(I_n \!-\! \gamma C)\Diag(x)^{\scriptscriptstyle 1/2} \preceq 0 \ (\text{all } \tau_i \le 0) 
\Rightarrow R(x)^{\scriptscriptstyle -1} \!-\! I_n \preceq 0.
\end{align*}
The result follows.
\qed
\end{proof}


\section{Investigating the forbidden region through the Hessian}\label{sec:forbidden}

A natural question arising from the preceding analysis concerns the behavior of the objective function in \ref{nlpid_bound} when
$\gamma$ is in the \emph{forbidden region}
$\textstyle \left(1/\lambda_{\max},1/\lambda_{\min}\right).
$

\subsection{Certifying bad $\gamma$}
In what follows, based on positive-definite $C$,
we give a set of intervals of  $\gamma$ for which the Hessian fails to be negative semidefinite at 
$x=\mathbf{e}$.
Consequently, the 
objective function of \ref{nlpid_bound} is not concave on at least one point in its domain, for such values of $\gamma$. 

\begin{theorem}\label{thm:badgamma}
Suppose that  $C\in \mathbb{S}^n_{++}$ and $\gamma>0$. 
    Let $R(x):=I_n +  \Diag(x)^{\scriptscriptstyle 1/2}(\gamma C-I_n)\Diag(x)^{\scriptscriptstyle 1/2}$ and $f(x):=\ldet(R(x))$, for $x \in\mathbb{R}^n_{++}$ such that $R(x) \succ 0$. Let $\hat{x}:=\mathbf{e}$. Then, $\nabla^{\scriptscriptstyle 2} f(\hat{x})\npreceq 0$, if for some $1\leq i < j\leq n$,     
\[
\gamma \ \in\ 
\begin{cases}
(\gamma_{\scriptscriptstyle L}, \gamma_{\scriptscriptstyle R}), & \mbox{if } (a-b)^{\scriptscriptstyle 2} < 4c^{\scriptscriptstyle 2}; \\ 
(\gamma_{\scriptscriptstyle L}, \gamma_{\scriptscriptstyle l}) \cup (\gamma_{\scriptscriptstyle r}, \gamma_{\scriptscriptstyle R}), & \mbox{if }(a-b)^{\scriptscriptstyle 2} \ge 4c^{\scriptscriptstyle 2},
\end{cases}
\]
where 
$
\gamma_{\scriptscriptstyle L,R} := \textstyle \frac{a+b \mp \Delta_1}{2}$,
$
\gamma_{\scriptscriptstyle l,r} := 
\textstyle \frac{a+b \mp \Delta_2}{2}$, 
$\Delta_1:=\sqrt{(a-b)^{\scriptscriptstyle 2}+4c^{\scriptscriptstyle 2}}$, $\Delta_2:=\allowbreak\sqrt{(a-b)^{\scriptscriptstyle 2}-4c^{\scriptscriptstyle 2}}$, 
$a := (C^{\scriptscriptstyle -1})_{ii}$\,, 
$b := (C^{\scriptscriptstyle -1})_{jj}$\,, 
$c := (C^{\scriptscriptstyle -1})_{ij}$\,.
\end{theorem} 

\begin{proof}
We have
$R(\mathbf{e}) = \gamma C$, 
and
\[
\nabla^{\scriptscriptstyle 2} f(\mathbf{e}) = 2\,\operatorname{Diag}\!\left(\textstyle\frac{1}{\gamma}C^{\scriptscriptstyle -1}\right) - I_n - \left(\textstyle\frac{1}{\gamma}C^{\scriptscriptstyle -1}\right)^{\scriptscriptstyle (2)}= -\left(\textstyle\frac{1}{\gamma}C^{\scriptscriptstyle -1} - I_n\right)^{\scriptscriptstyle (2)}.
\]    
It is evident that if
$
\textstyle\frac{1}{\gamma}C^{\scriptscriptstyle -1} - I_n \succeq 0$ or $\textstyle\frac{1}{\gamma}C^{\scriptscriptstyle -1} - I_n \preceq 0$,
the Hessian $\nabla^{\scriptscriptstyle 2} f(\mathbf{e})$ is negative semidefinite. 

Let $A := \textstyle\frac{1}{\gamma}C^{\scriptscriptstyle -1} - I_n$ and $M := A^{\scriptscriptstyle (2)}$,
so that $\nabla^{\scriptscriptstyle 2} f(\mathbf{e}) = -M$. 
Consider the $2\times 2$ principal minor of $M$:
\[
\det
\begin{pmatrix}
M_{ii} & M_{ij} \\
M_{ij} & M_{jj}
\end{pmatrix}
= A_{ii}^{\scriptscriptstyle 2} A_{jj}^{\scriptscriptstyle 2} - A_{ij}^{\scriptscriptstyle 4}\,.
\]
If this determinant is negative, $M$ cannot be positive semidefinite, and therefore $\nabla^{\scriptscriptstyle 2} f(\mathbf{e})$ is not negative semidefinite. This occurs precisely when
$
A_{ij}^{\scriptscriptstyle 2} > |A_{ii}|\,|A_{jj}|.
$

Because  $A_{ii} = (\textstyle\frac{1}{\gamma}C^{\scriptscriptstyle -1})_{ii} - 1$ and $A_{ij} = (\textstyle\frac{1}{\gamma}C^{\scriptscriptstyle -1})_{ij}$\,,
the condition becomes
\[
\textstyle\left(\textstyle\frac{1}{\gamma}C^{\scriptscriptstyle -1}\right)_{ij}^{\scriptscriptstyle 2} \;>\; \textstyle\left| \left(\textstyle\frac{1}{\gamma}C^{\scriptscriptstyle -1}\right)_{ii\vphantom{j}} - 1 \right| \cdot \textstyle\left| \textstyle\left(\textstyle\frac{1}{\gamma}C^{\scriptscriptstyle -1}\right)_{jj} - 1 \textstyle\right|,
\]
which reduces to
$c^{\scriptscriptstyle 2} > |a-\gamma|\,|b-\gamma| = |(a-\gamma)(b-\gamma)|$,
i.e., the two-sided quadratic inequality
\[
-c^{\scriptscriptstyle 2} \;<\; (a-\gamma)(b-\gamma) \;<\; c^{\scriptscriptstyle 2}.
\]

For the right-hand inequality, which
can be written as 
the convex inequality $\gamma^{\scriptscriptstyle 2} - (a+b)\gamma + (ab - c^{\scriptscriptstyle 2}) < 0$, we have:
\[
\textstyle\gamma \in (\gamma_{\scriptscriptstyle L}, \gamma_{\scriptscriptstyle R}), \quad
\gamma_{\scriptscriptstyle L,R} = \frac{(a+b) \mp \Delta_1}{2}, \quad
\Delta_1 := \sqrt{(a-b)^{\scriptscriptstyle 2} + 4c^{\scriptscriptstyle 2}}. 
\]

The left-hand inequality 
can be written as
the nonconvex inequality $\gamma^{\scriptscriptstyle 2} - (a+b)\gamma + (ab+c^{\scriptscriptstyle 2}) > 0$, with discriminant $\Delta_2^{\scriptscriptstyle 2} := (a-b)^{\scriptscriptstyle 2} - 4c^{\scriptscriptstyle 2}$. Then,
\begin{itemize}
    \item if $(a-b)^{\scriptscriptstyle 2} < 4c^{\scriptscriptstyle 2}$: $\Delta_2^{\scriptscriptstyle 2} < 0$, there are no real roots, and the inequality holds for all $\gamma$.
    \item if $(a-b)^{\scriptscriptstyle 2} \ge 4c^{\scriptscriptstyle 2}$: we have roots $\gamma_{l,r} = \textstyle\frac{(a+b)\mp\Delta_2}{2}$, $\Delta_2 := \sqrt{(a-b)^{\scriptscriptstyle 2} - 4c^{\scriptscriptstyle 2}}$, and the inequality holds for $\gamma \notin [\gamma_{\scriptscriptstyle l}, \gamma_{\scriptscriptstyle r}]$.
\end{itemize}
The result follows.
\qed
\end{proof}

With the following $n=2$ example, we verify that the interval from the first case of Theorem 
\ref{thm:badgamma} can be nonempty. 

\begin{example}
Let
$C :=\left(\begin{smallmatrix} 2 & 1 \\ 1 & 2 \end{smallmatrix}\right)$. We have $\lambda_{\min} = 1$, $\lambda_{\max} = 3$, and
$C^{\scriptscriptstyle -1} = \textstyle \frac{1}{3}\left(\begin{smallmatrix} 2 & -1 \\ -1 & 2 \end{smallmatrix}\right)$,
so $a = b = \frac{2}{3}$, $c = -\frac{1}{3}$.
As $(a-b)^{\scriptscriptstyle 2} = 0 < 4c^{\scriptscriptstyle 2} = \frac{4}{9}$,  we conclude that $\nabla^{\scriptscriptstyle 2}f(\mathbf{e})\npreceq 0$ for every $\gamma\in(\gamma_{\scriptscriptstyle L},\gamma_{\scriptscriptstyle R})= \left(\frac{1}{3},1\right)$.
This interval coincides exactly with the range of $\gamma$ for which $\frac{1}{\gamma}C^{\scriptscriptstyle -1}-I_n$ is indefinite. Thus, in this simple example, this necessary condition is also sufficient to guarantee the failure of $\nabla^{\scriptscriptstyle 2}f(\mathbf{e})\preceq 0$. 
\hfill $\clubsuit$
\end{example}

With the following $n=3$ example, we verify that the two intervals from the second case of Theorem 
\ref{thm:badgamma} can be nonempty, and we numerically check that there is an interval of $\gamma$ between these two intervals, where the Hessian is negative semidefinite on all of $(0,1]^3$.

\begin{example}\label{ex:three}
Consider
\[
C:=\left(\begin{smallmatrix} 2 & 1 & 0 \\ 1 & 2 & 1 \\ 0 & 1 & 6 \end{smallmatrix}\right) ~\mbox{ and }~\textstyle
C^{\scriptscriptstyle -1} = \left(\begin{smallmatrix}
11/16 & -3/8 & 1/16 \\
-3/8 & 3/4 & -1/8 \\
1/16 & -1/8 & 3/16
\end{smallmatrix}\right).
\]
 
\paragraph{Pair $(i,j)=(1,2)$:} $a = \tfrac{11}{16}$, $b = \tfrac{3}{4}$, $c = -\tfrac{3}{8}$.
\[
\begin{array}{c}
\textstyle
(a-b)^{\scriptscriptstyle 2} - 4c^{\scriptscriptstyle 2} = -\frac{143}{256} < 0 \Rightarrow  \text{first case (single interval).}\\[3pt] 
\textstyle
\gamma_{\scriptscriptstyle L} = \frac{23-\sqrt{145}}{32} \approx 0.34245, \quad
\gamma_{\scriptscriptstyle R} = \frac{23+\sqrt{145}}{32} \approx 1.09505.
\end{array}
\]
 
\paragraph{Pair $(i,j)=(1,3)$:} $a = \tfrac{11}{16}$, $b = \tfrac{3}{16}$, $c = \tfrac{1}{16}$.
\[
\begin{array}{c}
\textstyle
(a-b)^{\scriptscriptstyle 2} - 4c^{\scriptscriptstyle 2} = \frac{15}{64} > 0 \Rightarrow  \text{second case (two intervals).}\\[3pt]
\textstyle
\gamma_{\scriptscriptstyle L} = \frac{7-\sqrt{17}}{16} \approx 0.17981, \quad
\gamma_{\scriptscriptstyle l} = \frac{7-\sqrt{15}}{16} \approx 0.19544,\\[3pt]
\textstyle
\gamma_{\scriptscriptstyle r} = \frac{7+\sqrt{15}}{16} \approx 0.67956, \quad
\gamma_{\scriptscriptstyle R} = \frac{7+\sqrt{17}}{16} \approx 0.69519.
\end{array}
\]
 
\paragraph{Pair $(i,j)=(2,3)$:} $a = \tfrac{3}{4}$, $b = \tfrac{3}{16}$, $c = -\tfrac{1}{8}$.
\[
\begin{array}{c}
\textstyle
(a-b)^{\scriptscriptstyle 2} - 4c^{\scriptscriptstyle 2} = \frac{65}{256} > 0 \Rightarrow  \text{second case (two intervals).}\\[3pt]
\textstyle
\gamma_{\scriptscriptstyle L} = \frac{15-\sqrt{97}}{32} \approx 0.16097, \quad
\gamma_{\scriptscriptstyle l} = \frac{15-\sqrt{65}}{32} \approx 0.21680,\\[3pt]
\textstyle
\gamma_{\scriptscriptstyle r} = \frac{15+\sqrt{65}}{32} \approx 0.72070, \quad
\gamma_{\scriptscriptstyle R} = \frac{15+\sqrt{97}}{32} \approx 0.77653.
\end{array}
\]
 The forbidden region is $(1/\lambda_{\max}\,, 1/\lambda_{\min}) \approx (0.1600,\allowbreak 1.1145)$. The union of the five certified-bad intervals from the three pairs 
simplifies to
$
(0.16097,\,\allowbreak 0.21680) \,\cup\, (0.34245,\,\allowbreak 1.09505)$.
Additionally, 
We have obtained a lot of numerical evidence that the true region of $\gamma$ for which 
the Hessian is not negative semidefinite on all of $(0,1]^3$
is approximately 
$
(0.16029,\,\allowbreak 0.23825) \,\cup\, (0.31263,\,\allowbreak 1.09574).
$
We can see that Theorem
\ref{thm:badgamma} is certifying over $93\%$ of the ``true forbidden region''.
\hfill $\clubsuit$
\end{example}

\begin{remark}
Considering Theorem 
\ref{thm:badgamma}, we could also ask for conditions under which 
the objective function of \ref{nlpid_bound} is
not concave at a point (e.g., $(s/n)\mathbf{e}$)
in the relative interior of its feasible region.
For this, we would have to analyze a 4-th degree polynomial, and we would not get a clean result like Theorem \ref{thm:badgamma}, although we would still get computable intervals. Finally, we could try to go further and work with the Hessian 
of the objective function written
with respect to a basis for the reduced $(n\!-\!1)$-dimensional affine space $\{x\in \mathbb{R}^n \,:\, \mathbf{e}^\top x = s\}$, but this gets even messier.
\end{remark}


\subsection{Certifying good $\gamma$}

Next, through a stylized family of examples, we demonstrate that it is possible to certify that the objective function of the \ref{nlpid_bound} is concave and non-decreasing in $\gamma$, at some values of $\gamma$ in the forbidden region.
In the following two theorems,
$f(\cdot)$ and $R(\cdot)$
are as defined in Theorem \ref{thm:concave}, and 
$h(\cdot)$ and $\Theta(\cdot)$
are as defined in the proof of Lemma 
\ref{lem:natural_to_nlp_modified},
but with respect to the particular matrix $C$ that we will define. 

\begin{theorem}\label{thm:good}
For $n\ge 1$ and $0\!<\!a\!<
\!\tfrac13$,
let  $C:=(1-a)I_n+a\,\mathbf{e}\mathbf{e}^{\top}$ and 
$
\textstyle \gamma \in \left(0,\ \frac{1-3a}{(1-a)^{\scriptscriptstyle 2}}\right]$. 
Then,  $\nabla^{\scriptscriptstyle 2} f(x)\preceq 0$ for all $0\leq x\leq \mathbf{e}$.  Moreover, 
if
$a<\frac{n-2}{3n-2}$ (which implies that $n\geq 3$),
then the interval properly contains 
$(0,1/\lambda_{\max}]$\,.
\end{theorem}

\begin{proof}
Define
$
b:=\gamma(1-a)-1$ and $\kappa:=\gamma a$. Note that $b<0$ and $\kappa>0$. 

We have $\gamma C - I_n = bI_n + \kappa\,\mathbf{e}\mathbf{e}^{\top}$ and $R(x)=I_n+b\Diag(x)+\kappa\,x^{\scriptscriptstyle 1/2}(x^{\scriptscriptstyle 1/2})^{\top}$. As  $1+b=\gamma(1-a)>0$, we have $1+bx_i\ge 1+b>0$ for all $x_i\in(0,1]$. Hence, $R(x)\succ0$ for all $0<x\leq \mathbf{e}$. By the Sherman--Morrison formula,  $R(x)^{\scriptsize -1}= \Diag(d) - \tau ww^\top$, where,  
$
d_i:=1/(1+bx_i)$ and $w_i:=d_i(x_i)^{\scriptsize -1/2}$, for $i\in N$, and $\tau:=\kappa/(1+\kappa\sigma)>0$ and $\sigma:=\sum_{i\in N} d_ix_i$\,.

Moreover,  $(R(x)^{\scriptsize -1} -I_n)_{ii}=d_i -1 -\tau w_i^2$ and $(R(x)^{\scriptsize -1})_{ij}= -\tau w_iw_j$ $(i\neq j)$. Let $p_i:=\tau w_i^{\scriptsize 2}\geq 0$ and define  
$
M := -(R(x)^{\scriptscriptstyle -1}-I_n)^{\scriptscriptstyle (2)}=
-(pp^\top +\Diag(q))$, where $q_i:=(d_i-1-p_i)^{\scriptsize 2} - p_i^{\scriptsize 2}$.  

From Lemma \ref{lem:hes_f}, we have that a sufficient condition for $\nabla^{\scriptsize 2}f(x)\preceq 0$ is $M\preceq0$. Moreover, a sufficient condition for  $M\preceq0$ is
$q_i\geq0$ for every $i\in N$.
Equivalently,
$|d_i-1-p_i|\geq p_i$\,,
which holds if
$d_i\leq1$ or $d_i\geq1+2p_i$\,.

As $b<0$ and $x_i>0$, we have $d_i=1/(1+bx_i)>1$, so the first alternative cannot occur. It therefore suffices to show that
 $d_i\ge 1+2p_i$ for every $i$. Using $d_i-1=|b|x_id_i$ and $p_i=\tau d_i^2x_i$\,, this condition is equivalent, after dividing by $x_id_i>0$, to
$d_i\le |b|/2\tau$, for every  $i$.

Now, for every $0<x\leq \mathbf{e}$, we have 
$\tau\le\kappa$ and 
$d_i\le d^*:=\frac{1}{1+b}=\frac{1}{\gamma(1-a)}$. 
Hence, $\tau d_i\leq \kappa\,d^*$. Finally, 
$
 \kappa\,d^*\le \tfrac{|b|}{2}
\Leftrightarrow 2\kappa\le |b|(1+b)
\Leftrightarrow 2\gamma a \le (1-\gamma(1-a))\,\gamma(1-a) \Leftrightarrow \textstyle \gamma \ \le\textstyle \frac{1-3a}{(1-a)^2}\,.
$

For the moreover part, we have  
$1/\lambda_{\max}=\textstyle \frac{1}{1+(n-1)a}$, and it can be checked that $a<\frac{n-2}{3n-2}$ is equivalent to 
$
\textstyle \frac{1}{1+(n-1)a} < \frac{1-3a}{(1-a)^{\scriptscriptstyle 2}}$\,.
\qed
\end{proof}

\begin{theorem}\label{thm:nonincreasingex}
For $n> 1$, $0\!<\!a\!<\!\tfrac13$,
let  $C:=(1-a)I_n+a\,\mathbf{e}\mathbf{e}^{\top}$ and $0<s<n$.
Then,
 the {\rm\ref{nlpid_bound}}  bound is non-increasing in $\gamma$ on $\left(0,\ \frac{1-2a}{(1-a)^{\scriptscriptstyle 2}}\right]$.
 Moreover, for such a $\gamma$ in that interval for which some optimal solution of  
 {\rm\ref{nlpid_bound}} is not in $\{0,1\}^n$, the  {\rm\ref{nlpid_bound}} bound is decreasing at $\gamma$. 
\end{theorem}

\begin{proof}
Let   $0\leq \hat x\leq \mathbf{e}$, with $\mathbf{e}^\top \hat{x}=s$. 
From \eqref{derivativeh}, we see that it suffices to prove that 
\begin{equation*} 
      \textstyle\sum_{i \in N} (1-\hat{x}_i)  (I_n-\Theta(\gamma)^{\scriptscriptstyle -1})_{ii}\leq 0\,.
\end{equation*}
 We note that  for  fixed $\hat x$ and $\gamma$, $\Theta(\gamma)=R(\hat x)$. Then, in the remainder of this proof, we use the same definitions and similar arguments to those in Theorem \ref{thm:good}. Particularly, using the same arguments used to prove that $R(x)\succ 0$ in Theorem \ref{thm:good}, we can verify that   $\Theta(\gamma)\succ 0$, for all $\gamma \in \left(0,\ \frac{1-2a}{(1-a)^{\scriptscriptstyle 2}}\right]$. Moreover,  
$(\Theta(\gamma)^{\scriptsize -1} -I_n)_{ii}=d_i -1 -\tau d_i^2\hat x_i$\,, where  $d_i:=1/(1+b\hat{x}_i)$, for $i\in N$, and $b:=\gamma(1-a)-1<0$, $\tau:=\kappa/(1+\kappa\sigma)>0$,  $\kappa:=\gamma a$, and $\sigma:=\sum_{i\in N} d_i\hat{x}_i$\,.
If $\hat{x}_i=1$, then $1-\hat x_i=0$ (trivially), and if $\hat{x}_i=0$, we can verify that  $(I_n-\Theta(\gamma)^{\scriptscriptstyle -1})_{ii}=0$. So, in the following we assume $0<\hat x_i<1$.
 In such a case, we will prove that   
$(I_n-\Theta(\gamma)^{\scriptsize -1})_{ii}< 0$. Equivalently, we will prove that $d_i-1> \tau d_i^2\hat x_i$\,. Exactly as in the proof of Theorem \ref{thm:good},
$d_i-1=-b\hat x_id_i$\,. 
Then,
$
d_i-1> \tau d_i^2\hat x_i  \Leftrightarrow d_i<-b/\tau.
$
We have that  $\tau< \kappa$ and
$d_i< d^*:=1/(1+b)$ for all $0\leq \hat x\leq \mathbf{e}$. So, $\tau d_i<\kappa d^*$, and it suffices to have
$\kappa\,d^*\le -b \Leftrightarrow \kappa\le -b(1+b)$.
Substituting $\kappa=\gamma a$, $b=\gamma(1-a)-1$ and dividing by $\gamma>0$, we obtain
$a\le (1-a)-\gamma(1-a)^2
\Leftrightarrow
\gamma\ \le\ \textstyle\frac{1-2a}{(1-a)^2}\,$.
Finally, consider $\hat x$ as an optimal solution to \ref{nlpid_bound} for any given $\gamma\in\left(0,\ \frac{1-2a}{(1-a)^{\scriptscriptstyle 2}}\right]$. The result follows.  \qed   
\end{proof}

\begin{remark}
From Theorems \ref{thm:good} and \ref{thm:nonincreasingex}, we can see that for
    the interval $\left(\frac{1-3a}{(1-a)^{\scriptscriptstyle 2}},\ \frac{1-2a}{(1-a)^{\scriptscriptstyle 2}}\right]$, we have not established concavity of the \ref{nlpid_bound} objective function, but we have established monotonicity in $\gamma$. So, in this range, the \ref{nlpid_bound} offers an improved bound, but we cannot calculate it via convex programming. 
\end{remark}

\subsection{AR(1) correlation matrices}

Next, we aim to demonstrate numerically, with a small ($n=10$) example coming from a simple autoregressive stochastic process, that some values of $\gamma$ within the forbidden region
can yield a convex optimization problem, and  for such values of $\gamma$, the resulting \ref{nlpid_bound} bound can be strictly better than the bounds obtained by setting $\gamma=1/\lambda_{\max}$ or $\gamma=1/\lambda_{\min}$\,,
the best possible values outside of the forbidden region.

We begin, for $n\geq 2$, by
considering the general ``AR(1) correlation
matrix''\footnote{For all such matrices, the inverse has a tridiagonal closed form (see \cite{Kac1953KMS}), and therefore the associated \ref{MESP} instances can be solved in $\mathcal{O}(n^5)$ time (see \cite{ALTHANI2023120}).} (which is a type of symmetric Toeplitz matrix) $C$ defined by 
$C_{ij}=\rho^{|i-j|}$ for $i,j=1,\ldots,n$, $0<|\rho|<1$.
Writing $\delta:=1-\rho^{\scriptscriptstyle 2}$, the inverse is tridiagonal (see \cite{Kac1953KMS}):
\[
(C^{\scriptscriptstyle -1})_{ii}=
\begin{cases}
1/\delta, & i\in\{1,n\},\\
(1+\rho^{\scriptscriptstyle 2})/\delta, & 1<i<n,
\end{cases}
\qquad
(C^{\scriptscriptstyle -1})_{i,i+1}=-\rho/\delta.
\]
The forbidden region $
\textstyle\left(1/\lambda_{\max}\,,1/\lambda_{\min}\right)$ can be calculated easily: (1) Find the least and greatest of the $n$ distinct roots of $\tan(n\theta) = -\frac{(1-\rho^{\scriptscriptstyle 2})\sin(\theta)}{(1+\rho^{\scriptscriptstyle 2})\cos(\theta) - 2\rho}$.
(2) Plug those two roots into $\lambda \!:=\! \frac{1 - \rho^{\scriptscriptstyle 2}}{1 - 2\rho \cos(\theta) + \rho^{\scriptscriptstyle 2}}$. 
Generally, we have
$
\textstyle\frac{1 - |\rho|}{1 + |\rho|} \!<\! \lambda_{\min} \!<\! \lambda_{\max} \!<\! \textstyle\frac{1 + |\rho|}{1 - |\rho|},
$
with $\lambda_{\min}$ (resp., $\lambda_{\max}$)
converging to its lower (resp., upper) bound as $n\rightarrow \infty$.

Next, we will see what subregion of the forbidden region can be certified by Theorem \ref{thm:badgamma}
as having the Hessian $\nabla^{\scriptscriptstyle 2} f(\mathbf{e})$ being not negative semidefinite. 

\begin{corollary}\label{cor:ar1} 
For the AR(1) correlation matrix $C$ with $n\!\geq\! 4$, the Hessian $\nabla^{\scriptscriptstyle 2} f(\mathbf{e})$ is not negative semidefinite on
$
\textstyle\left( \frac{(2+\rho^{\scriptscriptstyle 2})-|\rho|\sqrt{\rho^{\scriptscriptstyle 2}+4}}{2(1-\rho^{\scriptscriptstyle 2})},
\frac{1+|\rho|+\rho^{\scriptscriptstyle 2}}{1-\rho^{\scriptscriptstyle 2}}
\right),
$
independent of $n$.
\end{corollary}

\begin{proof}
We confine our attention to the pairs $(i,j)=(2,3)$  and $(i,j)=(1,2)$; symmetry considerations 
and an analysis of non-adjacent pairs imply that we do not get a stronger result by considering additional pairs.

\smallskip
For $(i,j)=(2,3)$, we have 
$a=b=(1+\rho^{\scriptscriptstyle 2})/\delta$, $c=-\rho/\delta$, so $(a-b)^{\scriptscriptstyle 2}=0<4c^{\scriptscriptstyle 2}$, i.e., we are in the ``single interval'' case of Theorem~13, with
\[
\textstyle\gamma \in \mathcal{I}^{\scriptscriptstyle 2,3} := 
(\gamma_{\scriptscriptstyle L}^{\scriptscriptstyle 2,3},\gamma_{\scriptscriptstyle R}^{\scriptscriptstyle 2,3}):=
\left( \frac{1-|\rho|+\rho^{\scriptscriptstyle 2}}{1-\rho^{\scriptscriptstyle 2}},\ \frac{1+|\rho|+\rho^{\scriptscriptstyle 2}}{1-\rho^{\scriptscriptstyle 2}} \right).
\]

\smallskip
For $(i,j)=(1,2)$, we have $a=1/\delta$, $b=(1+\rho^{\scriptscriptstyle 2})/\delta$, $c=-\rho/\delta$. A short computation gives
$
\textstyle(a-b)^{\scriptscriptstyle 2}-4c^{\scriptscriptstyle 2}=\frac{\rho^{\scriptscriptstyle 2}(\rho^{\scriptscriptstyle 2}-4)}{(1-\rho^{\scriptscriptstyle 2})^{\scriptscriptstyle 2}}<0 \quad \text{for all } \rho\in(-1,1),
$
so we are again in the ``single interval'' case of Theorem~13, with
\[
\textstyle\gamma \in \mathcal{I}^{\scriptscriptstyle 1,2} :=
(\gamma_{\scriptscriptstyle L}^{\scriptscriptstyle 1,2},\gamma_{\scriptscriptstyle R}^{\scriptscriptstyle 1,2}):=
\left( \frac{(2+\rho^{\scriptscriptstyle 2})-|\rho|\sqrt{\rho^{\scriptscriptstyle 2}+4}}{2(1-\rho^{\scriptscriptstyle 2})},\ \frac{(2+\rho^{\scriptscriptstyle 2})+|\rho|\sqrt{\rho^{\scriptscriptstyle 2}+4}}{2(1-\rho^{\scriptscriptstyle 2})} \right).
\]

\noindent Claim: $ \mathcal{I}^{\scriptscriptstyle 2,3} \cup \mathcal{I}^{\scriptscriptstyle 1,2}$ is a single interval, from which the main result follows.

\noindent Proof of claim:
It suffices to establish that
$
\gamma_{\scriptscriptstyle L}^{\scriptscriptstyle 1,2} < \gamma_{\scriptscriptstyle L}^{\scriptscriptstyle 2,3} <
\gamma_{\scriptscriptstyle R}^{\scriptscriptstyle 1,2} < \gamma_{\scriptscriptstyle R}^{\scriptscriptstyle 2,3}\,,
$
which, putting over a common denominator and simplifying, is 
equivalent to $-\sqrt{\rho^{\scriptscriptstyle 2}+4} < -2 +|\rho| < \sqrt{\rho^{\scriptscriptstyle 2}+4} < 2 +|\rho|$, which is easily verified.
\qed \end{proof}

\begin{remark}
We could also consider AR($k$) correlation matrices $C$ of order $n\geq k+1$, for all $k\geq 1$, in which case $C^{-1}$ is $(2k+1)$-diagonal. It can be checked that
such $C^{-1}$ have at most $k(k+1)$ distinct order-2 determinants corresponding to non-diagonal submatrices (the only ones that can contribute non-empty intervals); and attaining that only when $n\geq 3k+1$. For such $C$, Theorem \ref{thm:badgamma} yields at most $2k(k+1)$ intervals, independent of $n$. In fact, it is quite possible
that one could prove that it is always only one interval after merging, as we did for $k=1$ in Corollary \ref{cor:ar1}. 
\end{remark}

\begin{example} 
Turning now to a concrete case, we set $\rho=0.3$ and $n=10$, and then we have
that the forbidden region for 
$\gamma$ is 
$
\textstyle\left(1/\lambda_{\max}\,,1/\lambda_{\min}\right) \approx 
(0.56157,\ 1.82808).
$
Through what we learned above for general AR(1) matrices, Theorem \ref{thm:badgamma}
certifies that the Hessian $\nabla^{\scriptscriptstyle 2} f(\mathbf{e})$ is not negative semidefinite on the interval $(0.81499, 1.52747)$.
Through some intense  experimentation, we found that the 
 ``true 
forbidden region'' for this instance seems to be
$\approx(0.71457,1.67049)$.
This means that Theorem
\ref{thm:badgamma} is certifying over $74\%$  of the true forbidden interval.

Concretely, specifically for 
$\gamma=0.70$ and $\gamma=1.68$ (which are in the forbidden region but not in the true forbidden region),
we obtained a lot of numerical evidence that the Hessian is truly negative semidefinite. 
Figure \ref{fig:variousgamma}, 
reports gaps to the optimal values of \ref{MESP}, varying $s$, for the 
key values of $\gamma$ that we identified.
We found that for 
$\gamma=0.70$,
the \ref{nlpid_bound} bound is substantially better than that obtained with $\gamma=1/\lambda_{\max}$\,,
and for 
$\gamma=1.68$,
the \ref{nlpid_bound} bound is substantially better than that obtained with $\gamma=1/\lambda_{\min}\,$.
Additionally, we observe that the best bound is attained for values of $\gamma$ that lie sometimes to the left and sometimes to the right of the ``true forbidden interval''. Finally, we  computed the \ref{nlpid_bound} bound for ten values of $\gamma$ in the  interval $[1/\lambda_{\max}\,,0.70]$ (resp., $[1.68,1/\lambda_{\min}]$). We observed that the monotonic behavior of the bound, established in Lemma \ref{lem:natural_to_nlp_modified} for $(0,1/\lambda_{\max}]$ (resp., $[1/\lambda_{\min}\,,+\infty)$), persists throughout these extended intervals of convexity.
\hfill$\clubsuit$
\end{example}

\begin{figure}[!ht]%
    \centering
    \includegraphics[width=1.01\linewidth]{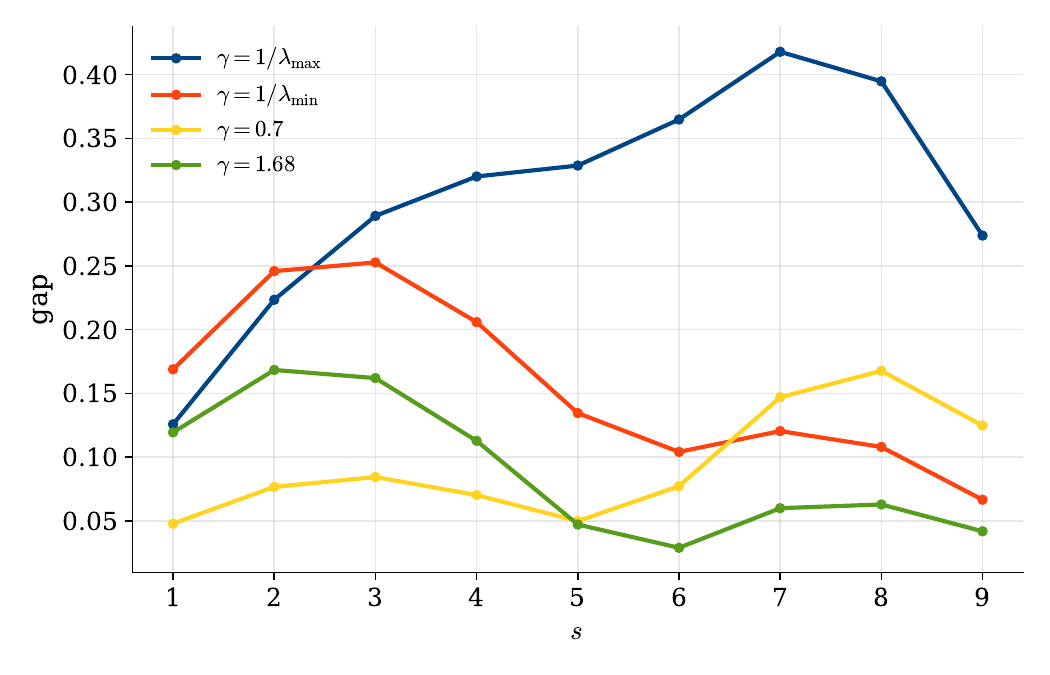} %
     \caption{Gaps for the \ref{nlpid_bound} bound}%
    \label{fig:variousgamma}%
\end{figure} 


\medskip
\noindent {\bf Acknowledgments.} 
This material is based upon work supported by the National Science Foundation under Grant No. DMS-1929284, carried out while the authors were in residence at the Institute for Computational and Experimental Research in Mathematics in Providence, RI, during the ``Maximum-Entropy Sampling: Mathematics and Algorithms'' Collaborate@ICERM program.

M.~Fampa was supported in part by CNPq grant 304340/2026-0.
J.~Lee was supported in part by AFOSR grant FA9550-22-1-0172.

\bibliographystyle{plain}
\bibliography{Bib}

\end{document}